\documentclass[a4paper]{amsart}

\usepackage[all]{xy}

\usepackage{amsmath}
\usepackage{amssymb}
\usepackage{amsfonts}
\usepackage{amsthm}
\usepackage{amscd}

\usepackage{mathrsfs}

\usepackage{mathtools}

\usepackage{pifont}
\usepackage{latexsym}

\usepackage{graphicx} 

\usepackage{stmaryrd}

\usepackage{comment}

\newcommand{\dalg}{\mathsf{Alg}}
\newcommand{\dfrm}{\mathsf{Frm}}

\newcommand{\vl}{\models}
\newcommand{\yy}{\rightarrow}

\newcommand{\power}{\mathcal{P}}

\newcommand{\nsystem}{\mathcal{N}}

\newcommand{\iand}{\bigwedge}
\newcommand{\ior}{\bigvee}

\newcommand{\ua}{\mbox{$\uparrow$}}

\newcommand{\thn}{\ \Rightarrow\ }
\newcommand{\eq}{\ \Leftrightarrow\ }

\newcommand{\gm}{\Gamma}
\newcommand{\dl}{\Delta}

\newcommand{\logic}[1]{\mathsf{#1}}

\newcommand{\barcan}{\forall x\Box\phi\supset\Box\forall x\phi}

\newcommand{\qflt}[1]{\mathcal{QF}_{#1}}

\newcommand{\nec}{\Box}

\newcommand{\nmodel}[1]{\mathit{#1}}
\newcommand{\nval}[1]{v_{#1}}
\newcommand{\nass}{\mathcal{A}}
\newcommand{\nint}{\mathcal{I}}
\newcommand{\ndom}{\mathcal{D}}
\newcommand{\ninta}{\mathcal{J}}

\newcommand{\lcong}[1]{\cong_{\logic{#1}}}
\newcommand{\lind}[1]{A(\logic{#1})}

\newcommand{\variable}{\mathsf{V}}

\newcommand{\predsym}{\mathsf{Pred}}

\newcommand{\prmt}{\mathrm{MT}}
\newcommand{\prtp}{\mathrm{TP}}
\newcommand{\prcf}{\mathrm{CF}}

\newcommand{\eko}{\mathsf{E}}
\newcommand{\cko}{\mathsf{C}}

\newcommand{\deffrm}{\mathscr{C}}
\newcommand{\defformula}{\mathscr{L}}

\newcommand{\qglo}{\logic{QGL}_{\Diamond^{\ast}}}

\newcommand{\nqgl}{\mathsf{NQGL}}
\newcommand{\psqglo}{\mathrm{PS}_{\logic{QGL}_{\Diamond^{\ast}}}}
\newcommand{\cglo}{\mathcal{C}_{\mathsf{GL}_{\Diamond^{\ast}}}}

\newcommand{\ccklm}{\mathcal{C}_{\mathsf{CKL}^{-}}}
\newcommand{\acklm}{\mathcal{A}_{\mathsf{CKL}^{-}}}

\newcommand{\psqckl}{\mathrm{PS}_{\logic{QCKL}}}
\newcommand{\psqcklm}{\mathrm{PS}_{\logic{QCKL}^{-}}}
\newcommand{\psck}{\mathrm{CK}}

\newcommand{\qcklm}{\logic{QCKL}^{-}}

\numberwithin{equation}{section}

\theoremstyle{plain}
\newtheorem{theorem}{Theorem}[section]
\newtheorem{lemma}[theorem]{Lemma}

\newtheorem{corollary}[theorem]{Corollary}

\theoremstyle{definition}
\newtheorem{definition}[theorem]{Definition}

\newtheorem{example}[theorem]{Example}
\date{}

\begin{document}
\title[Models for predicate modal 
logics with $\omega$-rules]
{Neighborhood and algebraic models for predicate modal logics with $\omega$-rules}
\author{Yoshihito Tanaka}{Kyushu Sangyo University}

\maketitle

\begin{abstract}
This paper investigates neighborhood and algebraic models for
predicate modal logics with $\omega$-rules, including non-normal cases. 
We establish sufficient conditions under which 
such logics
have neighborhood models with constant domains
and satisfy the completeness theorem with respect to 
neighborhood frames with constant domains. 
Related results for normal modal logics with $\omega$-rules  
were obtained 
by Tanaka \cite{tnknoncpt}, while similar results for 
non-normal modal logics without $\omega$-rules were presented 
by Arl\'{o}-Costa and Pacuit \cite{arlcst-pct06} and 
by Tanaka \cite{tnk22}. 
The results presented here extend these works. 
As applications, we 
prove that a predicate extension of $\logic{GL}$
is sound and complete with respect to a class of neighborhood frames 
with constant domains, 
and that a predicate common knowledge logic is Kripke incomplete but 
neighborhood complete.  
\end{abstract}

%%%%%%%%%%%%%%%%%%%%%%%%%%%%%%%%%%%%%%%%%%%%%%%%%%%%
%%%%%%%%%%%%%%%%%%%%%%%%%%%%%%%%%%%%%%%%%%%%%%%%%%%%
\section{Introduction}
%%%%%%%%%%%%%%%%%%%%%%%%%%%%%%%%%%%%%%%%%%%%%%%%%%%%
%%%%%%%%%%%%%%%%%%%%%%%%%%%%%%%%%%%%%%%%%%%%%%%%%%%%

This paper investigates neighborhood and algebraic models for
predicate modal logics with $\omega$-rules, including non-normal cases. 
We establish sufficient conditions under which 
such logics
have neighborhood models with constant domains
and satisfy the completeness theorem with respect to 
neighborhood frames with constant domains. 
Related results for normal modal logics with $\omega$-rules  
were obtained 
by Tanaka \cite{tnknoncpt}, while similar results for 
non-normal modal logics without $\omega$-rules were presented 
by Arl\'{o}-Costa and Pacuit \cite{arlcst-pct06} and 
by Tanaka \cite{tnk22}. 
The results presented here extend these works. 
As applications, we 
prove that a predicate extension of $\logic{GL}$
is sound and complete with respect to a class of neighborhood frames 
with constant domains, 
and that a predicate common knowledge logic is Kripke incomplete but 
neighborhood complete.

This paper discusses modal logics that admit $\omega$-rules, that 
is, inference rules with countably many premises.  
It is known that the use of $\omega$-rules plays a crucial 
role in the axiomatization of certain predicate modal logics, because 
natural predicate extensions of some propositional modal 
logics are not computably enumerable, even when 
the underlying propositional modal logics
are decidable. 
One example of this phenomenon is 
the provability logic $\logic{GL}$. 
It is well known that $\logic{GL}$ is decidable and that 
it defines 
the class $\mathcal{CW}$ of conversely well-founded Kripke frames  
and is sound and complete with respect to $\mathcal{CW}$. 
However, 
the predicate modal logic defined by the class $\mathcal{CW}$
is not computably enumerable 
\cite{ryb01,ryb24}. 
Another example is 
the common knowledge logic. 
The common knowledge logic is 
a multimodal logic with modal operators
$\cko$ and $\eko$, and is characterized by the class  $\mathcal{C}$ of Kripke 
frames such that $R_{\cko}=\bigcup_{n\in\omega}\left(R_{\eko}\right)^{n}$, 
where $R_{\cko}$ and $R_{\eko}$ denote accessibility relations 
corresponding to the modal operators $\cko$ and $\eko$, respectively. 
Although the propositional common knowledge logic is decidable 
\cite{hlp-mss92,mey-vdh95}, 
the predicate modal logic characterized by the class $\mathcal{C}$ of Kripke frames 
with constant domains is not computably enumerable \cite{wlt00}.

The model existence theorem 
for any propositional normal modal logic $\logic{L}$
with $\omega$-rules satisfying condition \eqref{countable} 
was established by Goldblatt \cite{gld93} and Segerberg \cite{sgb94}. 
Specifically, they proved that there exists a Kripke model $M$ satisfying 
\eqref{modelinfmeet} such that 
the set of formulas valid in $M$ is exactly $\logic{L}$. 
Subsequently, Tanaka \cite{tnknoncpt} showed that 
any predicate normal modal logic containing the Barcan formula $\barcan$ 
and 
$\omega$-rules satisfying \eqref{countable} 
has a Kripke model with constant domains that satisfies \eqref{modelinfmeet}. 
In this paper, we prove the existence of neighborhood models with constant domains
that satisfy \eqref{modelinfmeet}
for both normal and non-normal predicate modal logics with $\omega$-rules 
satisfying \eqref{countable},
without assuming the Barcan formula.

This paper uses neighborhood frames not only to interpret
non-normal modal logics, 
but also to provide constant domain semantics for predicate modal logics
without assuming 
the Barcan formula. 
It is well known that neighborhood frames serve as a semantic framework for 
non-normal modal logics, 
that is, modal logics that do not validate 
$\Box\top\equiv\top$ or $\Box(p\land q)\equiv \Box p\land\Box q$. 
If we define $\iand\emptyset$ as $\top$, this can be summarized 
by saying that 
neighborhood frames can provide semantics for modal logics 
without assuming distributivity of the modal operator over
finite conjunctions. 
In fact, neighborhood frames can be used to interpret 
infinitary modal logics in which the distributivity of the modal operator over 
conjunctions of cardinality $\kappa$ does not hold, 
for any infinite cardinal $\kappa$ \cite{mnr16,tnk21}. 
When considering constant domain semantics, 
the same techniques can be naturally applied to predicate modal logics,
because in the canonical construction of 
predicate modal logics, quantifiers are interpreted as infinitary conjunctions
and disjunctions.  
Indeed, neighborhood frames provide 
constant domain semantics for both normal and non-normal  predicate modal logics without 
assuming the Barcan formula \cite{arlcst-pct06, tnk22}. 
We extend these results to predicate modal logics with $\omega$-rules.

Our arguments rely on several properties of modal algebras, 
such as the Q-filters \cite{rsw-skr63} 
or a J\'{o}nsson-Tarski style representation for 
modal algebras and neighborhood frames \cite{tnk22}, that is, 
a representation that preserves countably many 
infinite meets and joins.
One of the main tools in our approach is the fact 
that every neighborhood frame 
is equivalent to its complex modal algebra as semantics 
for predicate modal formulas. 
This is a natural extension of the well-known relationship
between Kripke frames, their associated complex normal modal algebras, 
and propositional formulas (see, e.g., \cite{blc-rjk-vnm01}). 
Although this equivalence 
is a straightforward generalization of a familiar fact, 
we are not aware of any source in which it is stated explicitly. 
To make the argument self-contained, we provide a proof here.

The paper is organized as follows: 
In Section~\ref{sec:pre}, we recall basic properties. 
In Section~\ref{sec:eqv}, we show that each neighborhood 
frame and the complex modal algebra of it are equivalent 
as models of predicate modal formulas. 
In Section~\ref{sec:logic}, we introduce 
predicate modal logics with $\omega$-rules. 
In Section~\ref{sec:mex}, we show the model existence theorem 
and the completeness theorem. 
In Section~\ref{sec:glcompleteness},
we prove the completeness of a predicate extension of $\logic{GL}$ with respect 
to neighborhood frames with constant domains. 
In Section~\ref{sec:cklincompleteness},
we prove the existence of  
predicate and propositional common knowledge logics that 
are neighborhood complete but Kripke incomplete.

%%%%%%%%%%%%%%%%%%%%%%%%%%%%%%%%%%%%%%%%%%%%%%%%%%%%
%%%%%%%%%%%%%%%%%%%%%%%%%%%%%%%%%%%%%%%%%%%%%%%%%%%%
\section{Preliminaries}\label{sec:pre}
%%%%%%%%%%%%%%%%%%%%%%%%%%%%%%%%%%%%%%%%%%%%%%%%%%%%
%%%%%%%%%%%%%%%%%%%%%%%%%%%%%%%%%%%%%%%%%%%%%%%%%%%%

In this section, we establish the conventions for terminology 
and notation, and review some fundamental concepts. 
For simplicity, we discuss the unimodal case, 
but it is straightforward to generalize the results presented here 
to the case with countably many modal operators.

For any set $X$, we write $|X|$ for the cardinality of $X$.  
Let $\langle W,\leq\rangle$ be a partially ordered set.  
For any $X\subseteq W$, 
we denote the upward closure of $X$ by
$\ua X$. 
That is, 
\[
\ua X=\{w\in W\mid \exists x\in X(x\leq w)\}. 
\]
Let $f\colon A\yy B$ be a mapping from a set $A$ to a set $B$. 
For any set $X\subseteq A$ and $Y\subseteq B$, the sets 
$f\left[X\right]$ and $f^{-1}\left[Y\right]$ are defined as follows: 
\[
f\left[X\right]
=
\{f(x)\in B\mid x\in X\},\ \ 
f^{-1}\left[Y\right]
=
\{x\in A\mid f(x)\in Y\}. 
\]

\begin{definition}\label{neighborhoodframes}
A {\em neighborhood frame} is a pair
$\langle C, \nsystem\rangle$, where
$C$ is a nonempty set and 
$\nsystem$ maps each $c\in C$ to a subset $\nsystem(c)$ of $\power(C)$. 
A neighborhood frame 
$\langle C, \nsystem\rangle$ 
is said to be 
{\em monotonic}, 
{\em topped}, or
{\em closed under finite intersections}
(denoted by $\prmt$, $\prtp$, or $\prcf$, respectively),
if it satisfies the following conditions: 
\begin{description}
\item[MT]
for any $c\in C$, $\nsystem(c)$ is an upward closed subset of 
$\power(C)$ ordered by inclusion. That is, 
$\ua\nsystem(c)=\nsystem(c)$ for any $c\in C$;
\item[TP]
for any $c\in C$, $\nsystem(c)$ contains $C$; 
\item[CF]
for any $c\in C$, if $X,\ Y\in \nsystem(c)$ then 
$X\cap Y\in \nsystem(c)$.  
\end{description}

\begin{comment}
Let $Z_{1}=\langle C_{1},\nsystem_{1}\rangle$ and $Z_{2}=\langle C_{2},\nsystem_{2}\rangle$
be neighborhood frames. 
A mapping $f\colon C_{1}\yy C_{2}$ is called a 
{\em homomorphism of neighborhood frames} from $Z_{1}$ to $Z_{2}$, if for 
any $c\in C_{1}$ and $X\subseteq C_{2}$, 
$$
f^{-1}[X]\in \nsystem_{1}(c)\eq X\in \nsystem_{2}(f(c))
$$
holds.  
\end{comment}

A neighborhood frame $\langle C,\nsystem\rangle$ is called a {\em Kripke frame} if it 
satisfies $\prmt$ and the condition that 
$$
\bigcap \nsystem(c)\in\nsystem(c)
$$ 
for any $c\in C$. 
If $\langle C,\nsystem\rangle$ is a Kripke frame, 
we define the accessibility relation $R$ on $C$ by 
$(x,y)\in R
\eq
y\in\bigcap \nsystem(x)
$ 
for any $x$ and $y$ in $C$. 
\end{definition}

\begin{definition}
An algebra $\langle A;\lor,\land,-,\nec,0,1\rangle$ is 
a {\em modal  algebra}, 
if its reduct $\langle A;\lor,\land,-,0,1\rangle$ is a Boolean algebra 
and $\nec$ is a unary operator on $A$. 
A modal algebra is said to be {\em complete} if its underlying Boolean 
algebra is complete. 
A complete modal algebra is {\em completely multiplicative}, 
if 
$$
\iand_{x\in X}\Box x=\Box\iand_{x\in X} x
$$
holds, for any $X\subseteq A$. 
A modal algebra $A$ 
is said to be 
{\em monotonic}, 
{\em topped}, or  
{\em closed under finite intersections}
(denoted by $\prmt$, $\prtp$, or $\prcf$, respectively), 
if it satisfies the following conditions: 
\begin{description}
\item[MT]
for any $x$ and $y$ in $A$, 
$\nec (x\land y)\leq \nec x\land\nec y$; 
\item[TP]
$\nec 1=1$; 
\item[CF]
for any $x$ and $y$ in $A$, 
$
\nec x\land \nec y\leq\nec(x\land y)
$. 
\end{description}
Note that 
each of $\prmt$ and $\prcf$
can be defined by a single equation, and 
$\prmt$ is equivalent to 
\begin{equation*}
x\leq y\thn \nec x\leq\nec y.  
\end{equation*}
Let $A$ and $B$ be modal algebras. 
A mapping $f\colon A\yy B$ is 
a {\em homomorphism of modal algebras}, if $f$ is a homomorphism 
of Boolean algebras such that 
$
f(\nec x)=\nec f(x)
$
for any $x\in A$. 
\end{definition}

\begin{definition}
Let $A$ be a Boolean algebra. A nonempty subset $F\subseteq A$ is 
a {\em filter} of $A$, if it satisfies the following conditions: 
\begin{enumerate}
\item
$\ua F=F$;
\item
$x,\ y\in F\thn x\land y\in F$, for any $x$ and $y$ in $A$. 
\end{enumerate}
A filter is {\em proper} if $0\not\in F$. 
A proper filter $F$ is said to be {\em prime} if 
$x\lor y\in F$ implies that either $x\in F$ or $y\in F$
for all $x$, $y$ in $A$. 
\end{definition}

\begin{definition}
(Rasiowa-Sikorski \cite{rsw-skr63}). 
Let $A$ be a Boolean algebra and   
let $S\subseteq \power(A)$. 
A prime filter $F$ is said to be a {\em Q-filter for $S$}, if 
it satisfies the following condition: 
for any $X\in S$,  
if 
$\iand X\in A$ and 
$X\subseteq F$
then
$\iand X\in F$. 
That is, $F$ is closed under existing meets of sets in $S$. 
\end{definition}

We write $\qflt{S}(A)$ for the set of all 
Q-filters for $S$ in the algebra $A$. 
The following lemma is called the Rasiowa-Sikorski lemma 
\cite{rsw-skr63}.

\begin{lemma} \label{ba-ipft}
(Rasiowa-Sikorski \cite{rsw-skr63}). 
Let $A$ be a
Boolean algebra and  
let $S$ be a countable subset of $\power(A)$. 
Then,
for any $a$ and $b$ in $A$ with $a\not\leq b$,    
there exists a Q-filter $F$ for $S$ 
such that $a\in F$ and $b\not\in F$. 
\end{lemma}

\begin{definition}\label{qfltnfr}
Let $A$ be a modal algebra and 
$S\subseteq\power(A)$.  
A neighborhood frame $\langle C,\nsystem\rangle$ is called a 
{\em Q-filter neighborhood frame of $A$ for $S$} 
if 
$C=\qflt{S}(A)$ and $\nsystem$ satisfies the following:
\begin{equation}\label{nonmonotonic1}
\nsystem(F)\supseteq
\left\{
\{G\in\qflt{S}(A)\mid x\in G\}
\mid
\Box x\in F
\right\}
\end{equation}
and 
\begin{equation}\label{nonmonotonic2}
\nsystem(F)\cap
\left\{
\{G\in\qflt{S}(A)\mid x\in G\}
\mid
\Box x\not\in F
\right\}=\emptyset
\end{equation}
for every $F\in\qflt{S}(A)$. 
A Q-filter neighborhood frame of $A$ for $S$ with neighborhood system $\nsystem$
is denoted by $\dfrm_{S,\nsystem}(A)$. 
\end{definition}

\begin{lemma}
Suppose $A$ satisfies $\prmt$. 
Then, \eqref{nonmonotonic1} and \eqref{nonmonotonic2} are 
equivalent to the following condition: 
for each $F\in\qflt{S}(A)$, 
\begin{equation}\label{dfrmcond}
\Box^{-1}[F]=\bigcup_{X\in\nsystem(F)}\bigcap X. 
\end{equation}
\end{lemma}

\begin{proof}
Let $F\in\qflt{S}(A)$. Then, 
\begin{align*}
&\text{\eqref{nonmonotonic1} \& \eqref{nonmonotonic1}}\\
&\eq
\forall x\in A\left(
\Box x\in F
\eq
\{G\in\qflt{S}(A)\mid x\in G\}\in\nsystem(F)
\right)\\
&\eq
\forall x\in A\left(
\Box x\in F
\eq
\exists X\in \nsystem(F) 
\left(
X\subseteq\{G\in\qflt{S}(A)\mid x\in G\}
\right)
\right)
&& \text{($\prmt$)}\\
&\eq
\forall x\in A\left(
\Box x\in F
\eq
\exists X\in \nsystem(F) 
\left(
x\in\bigcap X
\right)
\right)\\
&\eq
\forall x\in A\left(
x\in\Box^{-1}[F]
\eq
x\in\bigcup_{X\in\nsystem(F)}\bigcap X
\right)\\ 
&\eq
\eqref{dfrmcond}. 
\end{align*}
\end{proof}

\begin{lemma}\label{canonicalsystem}
Let $A$ be a modal algebra and  
$S$ be a countable subset of $\power(A)$. 
Let 
\begin{equation*}\label{dfrmdef}
\nsystem(F)=
\left\{
\{G\in\qflt{S}(A)\mid x\in G\}
\mid
\Box x\in F
\right\}
\end{equation*}
for any $F\in\qflt{S}(A)$. 
Then, 
$\nsystem$ satisfies \eqref{nonmonotonic1} and \eqref{nonmonotonic2}.
If $A$ satisfies $\prmt$, then
\begin{equation}\label{monotonesystem}
\nsystem(F)=
\ua\left\{
\{G\in\qflt{S}(A)\mid x\in G\}
\mid
\Box x\in F
\right\}
\end{equation}
also satisfies 
\eqref{nonmonotonic1} and \eqref{nonmonotonic2}.
\end{lemma}

\begin{proof}
It is trivial that $\nsystem$ satisfies \eqref{nonmonotonic1}. 
We show \eqref{nonmonotonic2}. 
Suppose that $\Box x\not\in F$. 
If 
$\{G\in\qflt{S}(A)\mid x\in G\}\in\nsystem(F)$, there exists 
$y\in F$ such that $\Box y\in F$ and 
\begin{equation}\label{nflemma}
\{G\in\qflt{S}(A)\mid x\in G\}=\{G\in\qflt{S}(A)\mid y\in G\}.
\end{equation}
Since $x\not=y$ and $S$ is countable, Lemma \ref{ba-ipft} implies that 
there exists $G\in\qflt{S}(A)$ such that either 
$x\in G$ and $y\not\in G$, or 
$x\not\in G$ and $y\in G$. This contradicts \eqref{nflemma}.

Suppose that 
$A$ satisfies $\prmt$. 
We show \eqref{nonmonotonic2}. 
Suppose that $\Box x\not\in F$ and 
$\{G\in\qflt{S}(A)\mid x\in G\}\in\nsystem(F)$. 
Then, there exists $\Box y\in F$ such that 
\begin{equation}\label{canonicitylemma}
\{G\in\qflt{S}(A)\mid y\in G\}\subseteq
\{G\in\qflt{S}(A)\mid x\in G\}. 
\end{equation}
Since $A$ is monotone, 
$y\not\leq x$.
By Lemma \ref{ba-ipft}, there exists 
$G\in\qflt{S}(A)$ such that 
$x\not\in G$ and 
$y\in G$. 
This contradicts \eqref{canonicitylemma}. 
\end{proof}

\begin{theorem}(Tanaka \cite{tnk22}).\label{matonfr}
Let $A$ be a  modal algebra,  
$S$ be a countable subset of $\power(A)$, 
and $\nsystem$ be a neighborhood system 
defined in Lemma \ref{canonicalsystem}
(if $A$ is monotonic, $\nsystem$ is defined by \eqref{monotonesystem}). 
If $A$ satisfies properties among 
$\prmt$, $\prtp$, and $\prcf$, 
then each Q-filter neighborhood frame 
$\langle\qflt{S}(A),\nsystem\rangle$
of $A$ for $S$ 
satisfies the properties corresponding to those of $A$. 
\end{theorem}

\begin{definition}\label{nfrma}
(Do\v{s}en \cite{dsn89}).
Let $Z=\langle C,\nsystem\rangle$ be a
neighborhood frame. 
Define the {\em complex modal algebra} of $Z$, which is denoted by
$\dalg(Z)$,  by 
$$
\dalg(Z)=
\langle
\power(C);\cup,\cap,C\setminus-,\nec_{Z},\emptyset,C
\rangle, 
$$
where
$$
\nec_{Z}X=\{c\in C\mid X\in\nsystem(c)\}
$$
for any $X\subseteq C$.
\end{definition}

\begin{lemma}(Do\v{s}en \cite{dsn89}, Tanaka \cite{tnk22}).\label{nfrtoma} 
Let $Z=\langle C,\nsystem\rangle$ be a
neighborhood frame. 
If $Z$ satisfies 
properties among 
$\prmt$, $\prtp$, and $\prcf$,  
then 
the complex modal algebra $\dalg(Z)$ of $Z$
satisfies the properties corresponding to those of $Z$. 
\end{lemma}

If $Z$ is a Kripke frame, then $\dalg(Z)$ is completely multiplicative.

\begin{theorem}(Tanaka \cite{tnk22}).\label{extjt}
Let $A$ be a modal algebra and  
$S$ be a countable subset of  $\power(A)$. 
Let $Z=\dfrm_{S,\nsystem}(A)$ be a Q-filter neighborhood frame of $A$ for $S$. 
Define 
$
f\colon  A\yy \dalg(Z)
$
by 
\begin{equation}\label{embedding}
f(x)=\{F\in\qflt{S}(A)\mid x\in F\},  
\end{equation}
for any $x\in A$. 
Then, $f$
is a monomorphism of modal algebras 
and satisfies 
\begin{equation}\label{finfmeet}
f\left(\iand X\right)=\iand f[X],
\end{equation}
for any $X\in S$ such that $\iand X\in A$. 
Suppose that $\nsystem$ is a neighborhood system 
defined in Lemma \ref{canonicalsystem}
(if $A$ is monotonic, $\nsystem$ is defined by \eqref{monotonesystem}). 
Then, if $A$
satisfies properties among $\prmt$, $\prtp$, and $\prcf$, 
then
$\dalg(\dfrm_{S,\nsystem}(A))$ satisfies the same properties as $A$. 
\end{theorem}

\begin{proof}
We prove only that $f(\Box x)=\Box_{Z} f(x)$ for every $x\in A$. 
Then, 
\begin{align*}
F\in f(\Box x)
&\eq
\Box x\in F\\
&\eq
f(x)\in\nsystem(F)
&&
\text{\eqref{nonmonotonic1} \& \eqref{nonmonotonic2}}\\
&\eq
F\in \Box_{Z}f(x)
\end{align*}
For the remainder of the proof, see \cite{tnk22}. 
\end{proof}

\begin{corollary}
Let $A$ be a modal algebra, 
$S$ be a countable subset of  $\power(A)$, 
and $C=\qflt{S}(A)$. 
Let $Z=\langle C,\nsystem\rangle$ be a neighborhood frame 
and 
$
f\colon  A\yy \dalg(Z)
$ be a function defined by \eqref{embedding}. 
Then, $\nsystem$ satisfies 
\eqref{nonmonotonic1} and \eqref{nonmonotonic2} if and only if 
$f(\Box x)=\Box_{Z}f(x)$ for every $x\in A$.
\end{corollary}

\begin{proof}
Suppose that $x\in A$ and $f(\Box x)=\Box_{Z}f(x)$. Then, 
\begin{align*}
\Box x\in F
&\eq
F\in f(\Box x)\\
&\eq
F\in \Box_{Z}f(x)\\
&\eq
f(x)\in\nsystem(F). 
\end{align*}
Hence, 
\eqref{nonmonotonic1} and \eqref{nonmonotonic2} hold. 
The converse is  proved in the proof of Theorem \ref{extjt}. 
\end{proof}

%%%%%%%%%%%%%%%%%%%%%%%%%%%%%%%%%%%%%%%%%%%%%%%%%%%%
%%%%%%%%%%%%%%%%%%%%%%%%%%%%%%%%%%%%%%%%%%%%%%%%%%%%
\section{Equivalence of neighborhood frames 
and their complex modal algebras}\label{sec:eqv}
%%%%%%%%%%%%%%%%%%%%%%%%%%%%%%%%%%%%%%%%%%%%%%%%%%%%
%%%%%%%%%%%%%%%%%%%%%%%%%%%%%%%%%%%%%%%%%%%%%%%%%%%%

In this section, we show that each neighborhood 
frame and the complex modal algebra of it are equivalent 
as models of predicate modal formulas.

\begin{definition}
The language of predicate modal formulas consists 
of the following symbols:  
\begin{enumerate} 
\item
a countable set $\variable$ of variables;

\item
$\top$ and $\bot$;

\item the logical connectives: 
$\land$,  $\neg$;

\item
the quantifier:
$\forall$;

\item
for each $n\in\omega$,   
a countable set $\predsym(n)$ of 
predicate symbols of arity $n$ 
(a predicate symbol of arity $0$ is called a {\em propositional variable});

\item
the modal operator:
$\nec$.
\end{enumerate}
\end{definition}

\begin{definition}
The set $\Phi$ of predicate modal formulas is 
the smallest set that satisfies:

\begin{enumerate}
\item
$\top$ and $\bot$ are in $\Phi$;

\item 
if $P\in\predsym(n)$ 
and 
$x_{1},\ldots,x_{n}\in\variable$ 
then
$P(x_{1},\ldots,x_{n})\in\Phi$, 
for each $n\in\omega$;

\item 
if $\phi$ and $\psi$ are in $\Phi$
then $(\phi\land\psi)\in \Phi$;

\item 
if $\phi\in\Phi$ 
then $(\neg\phi)$ and 
$(\nec\phi)$  are in $\Phi$;

\item
if $\phi\in\Phi$ 
and $x\in\variable$ 
then $(\forall x\phi)\in\Phi$. 
\end{enumerate}
\end{definition}

An {\em atomic formula} is a formula of the form
$P(x_{1},\ldots,x_{n})$ for some $n\in\omega$, 
$P\in\predsym(n)$, and 
$x_{1},\ldots,x_{n}\in\variable$.
We use the standard definition of free and bound variables for formulas. 
A formula $\phi$ is said to be {\em closed} if it contains no free 
variables.

The symbols $\lor$, $\supset$,  and $\exists$ are defined in the usual way.  
We write 
$\phi\equiv\psi$  and $\Diamond\phi$ to abbreviate
$(\phi\supset\psi)\land(\psi\supset\phi)$ and $\neg\nec\neg\phi$, 
respectively. 
For each formula $\phi$ and $n\in\omega$, 
$\Box^{0}\phi$ denotes $\phi$ and $\Box^{n+1}\phi$ denotes $\Box(\Box^{n}\phi)$. 
We define $\Diamond^{n}\phi$ in the same way. 
Let 
$x$ and $y$
be variables. 
For each formula $\phi$, we define
$[y/x]\phi$ as
the instance of substituting 
$y$ for the 
free occurrences of 
$x$ in $\phi$. 
We write $p$, $q,\ldots$ for propositional variables.

\begin{definition}
A {\em neighborhood model} for predicate modal logics is a 4-tuple 
$\langle C,\nsystem,\ndom,\nint\rangle$, 
where $\langle C,\nsystem\rangle$ is a neighborhood frame, 
$\ndom$ is a nonempty set called the {\em domain}, 
and 
$\nint$ is a mapping 
called the {\em interpretation} such that 
for each $n\in\omega$,  each $P\in\predsym(n)$, and each $c\in C$, 
$\nint$ maps $(c,P)$ to an $n$-ary relation
$P^{\nint}(c)\subseteq \ndom^{n}$ over $\ndom$.  
An {\em assignment} $\nass$ to $\ndom$ 
is a mapping from $\variable$ to 
$\ndom$. 
For any assignment $\nass$, 
any variable 
$x$, 
and any 
$d\in\ndom$,  
define an assignment $[d/x]\nass$
as follows: 
$$
[d/x]\nass(z)
=
\begin{cases}
\nass(z) & \text{if $z\not= x$}\\
d & \text{if $z= x$} 
\end{cases}.
$$
For each neighborhood model 
$\nmodel{M}=\langle C,\nsystem,\ndom,\nint\rangle$ and 
each assignment $\nass$, 
the valuation $\nval{\nint,\nass}$ of a formula 
$\phi\in\Phi$ on $\nmodel{M}$ is 
defined inductively as follows:
\begin{enumerate}
\item
$\nval{\nint,\nass}(\top)=C$, 
$\nval{\nint,\nass}(\bot)=\emptyset$; 

\item 
$
\nval{\nint,\nass}(P(x_{1},\ldots,x_{n}))
=
\{c\mid 
(\nass(x_{1}),\ldots,\nass(x_{n}))\in P^{\nint}(c)\}$, 
for any $n\in\omega$, 
$P\in\predsym(n)$, and  $x_{1},\ldots,x_{n}\in\variable$;

\item 
$\nval{\nint,\nass}(\phi\land\psi)=
\nval{\nint,\nass}(\phi)\cap\nval{\nint,\nass}(\psi)$;

\item 
$\nval{\nint,\nass}(\neg\phi)=
C\setminus \nval{\nint,\nass}(\phi)$;

\item 
$\nval{\nint,\nass}(\forall x\phi)=
\bigcap_{d\in \ndom}\nval{\nint,[d/x]\nass}(\phi)$;

\item 
$\nval{\nint,\nass}(\nec\phi)=
\{c\mid\nval{\nint,\nass}(\phi)\in\nsystem(c)\}$. 
\end{enumerate}
A neighborhood model $\nmodel{M}=\langle C,\nsystem,\ndom,\nint\rangle$ 
is said to satisfy $\prmt$, $\prtp$, or $\prcf$, if 
its underlying neighborhood frame $\langle C,\nsystem\rangle$  
satisfies the corresponding property. 
\end{definition}

Let 
$\nmodel{M}=\langle C,\nsystem,\ndom,\nint\rangle$ be a neighborhood model, 
let $\nass$ be an assignment for $\nmodel{M}$, and 
let $\phi\in\Phi$. 
For all assignments $\nass'$ for $\nmodel{M}$, 
if 
$
\nass(x)=\nass'(x)
$
for every free variable $x$ in $\phi$,  
then $\nval{\nint,\nass}(\phi)=\nval{\nint,\nass'}(\phi)$. 
Hence, if $\phi$ has no free variables
then $\nval{\nint,\nass}(\phi)=\nval{\nint,\nass'}(\phi)$ 
for any assignments $\nass$ and $\nass'$ for $\nmodel{M}$.

Let 
$\nmodel{M}=\langle C,\nsystem,\ndom,\nint\rangle$ be a neighborhood model.  
For any $\phi\in\Phi$ and $c\in C$, 
we write 
$
c\vl_{\nmodel{M}}\phi
$,
if 
$c\in\nval{\nint,\nass}(\phi)$ 
for any assignment $\nass$.   
If 
$
c\vl_{\nmodel{M}}\phi
$
for every $c\in C$, 
we write 
$
\nmodel{M}\vl\phi
$.  
Let 
$Z=\langle C,\nsystem\rangle$ be a neighborhood frame. 
We write 
$
Z\vl\phi
$,   
if 
for any domain $\ndom$ and any interpretation $\nint$, 
the neighborhood model $\nmodel{M}=\langle C,\nsystem,\ndom,\nint\rangle$ 
satisfies 
$
\nmodel{M}\vl\phi
$. 
Let $\gm$ be a set of formulas. 
If 
$
Z\vl\phi
$ 
for any $\phi\in\gm$, 
we write 
$
Z\vl\gm
$.
Let $\mathcal{C}$ be a class of neighborhood frames. 
We write $\mathcal{C}\vl\phi$ if $Z\vl\phi$ for every $Z\in \mathcal{C}$, 
and write 
$\mathcal{C}\vl\gm$ if $\mathcal{C}\vl\phi$ for every $\phi\in\gm$. 
We write $\deffrm(\gm)$ for the class of neighborhood frames defined by 
$$
\deffrm(\gm)
=
\{
Z\mid
Z\vl\gm
\}.
$$
Let $\mathcal{C}$ be a class of neighborhood frames. 
We write $\defformula(\mathcal{C})$ for the set of formulas defined by 
$$
\defformula(\mathcal{C})
=
\{
\phi\mid
\mathcal{C}\vl\phi
\}.
$$

\begin{definition}
An {\em algebraic model} for predicate modal logics 
is a triple $\langle A,\ndom,\ninta\rangle$, where $A$ is a complete 
modal algebra, $\ndom$ is a nonempty set, and $\ninta$ maps each 
$n$-ary predicate symbol to a mapping $P^{\ninta}\colon \ndom^{n}\yy A$. 
Let $\nass$ be an assignment to $\ndom$. 
The function $u_{\ninta,\nass}$ from the set $\Phi$ of formulas to $A$ 
is defined inductively as follows: 
\begin{enumerate}
\item
$u_{\ninta,\nass}(\top)=1$, $u_{\ninta,\nass}(\bot)=0$;

\item
$u_{\ninta,\nass}(P(x_{1},\ldots,x_{n}))
=
P^{\ninta}(\nass(x_{1}),\ldots,\nass(x_{n}))
$
for any 
$n\in\omega$, 
$P\in\predsym(n)$, and  $x_{1},\ldots,x_{n}\in\variable$;

\item
$u_{\ninta,\nass}(\phi\land\psi)=u_{\ninta,\nass}(\phi)\land u_{\ninta,\nass}(\psi)$;

\item
$u_{\ninta,\nass}(\neg\phi)=-u_{\ninta,\nass}(\phi)$;

\item
$u_{\ninta,\nass}(\forall x\phi)=\iand_{d\in\ndom}u_{\ninta,[d/x]\nass}\left(\phi\right)$; 

\item
$u_{\ninta,\nass}(\Box\phi)=\Box u_{\ninta,\nass}(\phi)$. 
\end{enumerate}
\end{definition}

Let 
$A$ be a complete modal algebra. 
We write 
$
A\vl\phi
$,   
if 
$u_{\ninta,\nass}(\phi)=1$ for any $\ndom$, $\ninta$, and $\nass$. 
Let $\gm$ be a set of formulas. We write 
$
A\vl\gm
$, 
if 
$
A\vl\phi
$ 
for any $\phi\in\gm$. 
Other semantic terminology and notation 
for algebraic models 
are defined in the same way 
as neighborhood models.

\begin{lemma}\label{nmodeltoamodel}
Let $Z=\langle C,\nsystem\rangle$ be a neighborhood frame and 
$\nmodel{M}=\langle C,\nsystem,\ndom,\nint\rangle$ be a neighborhood model. 
Define an algebraic model $\nmodel{A}^{\nmodel{M}}=\langle \dalg(Z),\ndom,\ninta\rangle$ by 
\begin{equation}\label{dualint}
c\in P^{\ninta}(d_{1},\ldots,d_{n})
\eq
(d_{1},\ldots,d_{n})\in P^{\nint}(c). 
\end{equation}
for any n-ary predicate symbol $P$. 
Then, 
$u_{\ninta,\nass}(\phi)=\nval{\nint,\nass}(\phi)$ for any formula $\phi$ and 
any assignment $\nass$. 
\end{lemma}

\begin{proof}
Induction on $\phi$. 
The case $\phi=P(x_{1},\ldots,x_{n})$: 
\begin{align*}
u_{\ninta,\nass}(P(x_{1},\ldots,x_{n}))
&=
P^{\ninta}(\nass(x_{1}),\ldots,\nass(x_{n}))\\
&=
\left\{c\mid
\left(\nass(x_{1}),\ldots,\nass(x_{n})\right)\in P^{\nint}(c)\right\}\\
&=
v_{\nint,\nass}(P(x_{1},\ldots,x_{n})). 
\end{align*}

\noindent
The case $\phi=\forall x\psi$: 
\begin{align*}
u_{\ninta,\nass}(\forall x\psi)
&=
\bigcap_{d\in\ndom}
u_{\ninta,[d/x]\nass}
\left(\psi\right)\\
&=
\bigcap_{d\in\ndom}
v_{\nint,[d/x]\nass}
\left(\psi\right)
&&
\text{(by induction hypothesis)}\\
&=
\nval{\nint,\nass}(\forall x\psi).
\end{align*}

\noindent
The case $\phi=\Box \psi$: 
\begin{align*}
u_{\ninta,\nass}(\Box \psi)
&=
\{c\in C\mid u_{\ninta,\nass}(\psi)\in\nsystem(c)\}\\
&=
\{c\in C\mid v_{\nint,\nass}(\psi)\in\nsystem(c)\}
&&
\text{(by induction hypothesis)}\\
&=
\nval{\nint,\nass}(\Box\psi).
\end{align*}
Other cases are straightforward. 
\end{proof}

\begin{lemma}\label{amodeltonmodel}
Let $Z=\langle C,\nsystem\rangle$ be a neighborhood frame
and 
$\nmodel{A}=\langle \dalg(Z),\ndom,\ninta\rangle$ be an algebraic model. 
Let 
$\nmodel{M}^{\nmodel{A}}=\langle C,\nsystem,\ndom,\nint\rangle$
be a neighborhood model, where $\nint$ is defined by \eqref{dualint}. 
Then, 
$\nval{\nint,\nass}(\phi)=u_{\ninta,\nass}(\phi)$ 
for any formula $\phi$ and 
any assignment $\nass$. 
\end{lemma}

\begin{proof}
Induction on $\phi$.  
\end{proof}

\begin{theorem}\label{modelduality}
Let $Z=\langle C,\nsystem\rangle$ be a neighborhood frame. 
For any formula $\phi$, 
$$
Z\vl\phi
\eq
\dalg(Z)\vl\phi. 
$$
\end{theorem}

\begin{proof}
First, suppose $Z\not\vl\phi$. Then, there exists a domain $\ndom$, 
an interpretation $\nint$, an assignment $\nass$, and 
$c\in C$ such that $c\not\in \nval{\nint,\nass}(\phi)$. 
By Lemma \ref{nmodeltoamodel}, 
$c\not\in u_{\ninta,\nass}(\phi)$. 
Hence, $u_{\ninta,\nass}(\phi)\not=C$. 
Therefore, $\dalg(Z)\not\vl\phi$.  
Next, suppose $\dalg(Z)\not\vl\phi$. Then, there exists a domain $\ndom$, 
an interpretation $\ninta$,  an assignment $\nass$, and 
$c\in C$ such that 
$c\not\in u_{\ninta,\nass}(\phi)$. 
Then, 
$
c\not\in \nval{\nint,\nass}(\phi)
$
by Lemma \ref{amodeltonmodel}. 
Hence, $Z\not\vl\phi$.  
\end{proof}

%%%%%%%%%%%%%%%%%%%%%%%%%%%%%%%%%%%%%%%%%%%%%%%%%%%%
%%%%%%%%%%%%%%%%%%%%%%%%%%%%%%%%%%%%%%%%%%%%%%%%%%%%
\section{Predicate modal logics with $\omega$-rules}\label{sec:logic}
%%%%%%%%%%%%%%%%%%%%%%%%%%%%%%%%%%%%%%%%%%%%%%%%%%%%
%%%%%%%%%%%%%%%%%%%%%%%%%%%%%%%%%%%%%%%%%%%%%%%%%%%%

In this section, we introduce predicate modal logics with $\omega$-rules. 
We define a logic as a set of formulas satisfying certain closure properties, 
rather than by formal systems.

\begin{definition}\label{pmlogic}
A set $\logic{L}\subseteq\Phi$ is called 
a {\em predicate modal logic}, if it 
contains all classical predicate tautologies, is closed under modus ponens, 
uniform substitution of formulas (see \cite{gbb-skv-shh12}), 
and satisfies the following conditions: 
\begin{enumerate}
\item
for any $\phi\in\Phi$ and $x\in\variable$, 
if $\phi\in\logic{L}$ then $\forall x\phi\in\logic{L}$;
\item\label{congrule}
for any $\phi$ and $\psi$ in $\Phi$, 
if $\phi\equiv\psi\in\logic{L}$, then $\nec\phi\equiv\nec\psi\in\logic{L}$. 
\end{enumerate}
A predicate modal logic $\logic{L}$ 
is said to be 
{\em monotonic}, 
{\em topped}, or  
{\em closed under finite intersections},  
which are denoted by $\prmt$, $\prtp$, or $\prcf$, 
if 
$\nec(p\land q)\supset\nec p\land \nec q\in\logic{L}$,
$\nec\top\in\logic{L}$,
or
$\nec p\land\nec q\supset\nec(p\land q)\in\logic{L}$,  respectively. 
A predicate  modal logic $\logic{L}$ is {\em normal} if it satisfies 
$\prmt$, $\prtp$, and $\prcf$. 
A predicate modal logic $\logic{L}$ is {\em consistent} if $\bot\not\in\logic{L}$. 
Let $\mathcal{C}$ be a class of neighborhood frames. 
A predicate modal logic $\logic{L}$ is {\em sound} with respect to $\mathcal{C}$
if $\logic{L}\subseteq\defformula(\mathcal{C})$, 
and {\em complete} with respect to $\mathcal{C}$ if the converse holds. 
A predicate modal logic $\logic{L}$ is said to be {\em neighborhood 
complete}, if there exists a class $\mathcal{C}$ of neighborhood frames 
such that $\logic{L}$ is sound and complete with respect to $\mathcal{C}$. 
It is clear that 
a predicate modal logic $\logic{L}$ is neighborhood complete 
if and only if 
$$
\logic{L}=\defformula(\deffrm(\logic{L})). 
$$
\end{definition}

We consider predicate modal logics that can be 
axiomatized by 
countably many pairs consisting of a set of axiom schemata and an 
$\omega$-rule of the following form: 
\begin{equation}\label{omega-rule}
\alpha\supset\beta_{i}\ \text{($i\in\omega$)},\ \hspace{10pt} 
\dfrac{p\supset\beta_{i}\ \text{($i\in\omega$)}}{p\supset\alpha}, 
\end{equation}
where $\beta_{i}$ ($i\in\omega$) and $\alpha$ are closed formulas. 
Intuitively, \eqref{omega-rule} is a proof-theoretic expression 
of the formula  $\alpha\equiv\iand_{i\in\omega}\beta_{i}$ of infinitary 
modal logic. 
Note that the axioms and the $\omega$-rule in \eqref{omega-rule} are 
schemata.
That is, the instances of any uniform substitution of formulas for 
atomic formulas in \eqref{omega-rule} 
are also considered to be axioms and inference rules. 
More precisely, 
let $\mathsf{Sub}$ be a set of functions from $\Phi$ to $\Phi$ that 
uniformly substitute formulas of $\Phi$ 
for atomic formulas. 
A predicate modal logic $\logic{L}$ is said to {\em admit} a pair of the form
\eqref{omega-rule} if it satisfies the following conditions: 
\begin{enumerate}
\item
for each $s\in\mathsf{Sub}$, 
$s(\alpha\supset\beta_{i})\in\logic{L}$ for any $i\in\omega$;
\item
for each $s\in\mathsf{Sub}$ and  
each formula $\phi$,
$\phi\supset s(\alpha)\in\logic{L}$ whenever
$\phi\supset s(\beta_{i})\in\logic{L}$ for every $i\in\omega$. 
\end{enumerate}
Throughout this paper, we restrict our attention to pairs of the form 
\eqref{omega-rule}
that satisfy the following 
condition: 
\begin{equation}\label{countable}
\left|\{
\{s(\beta_{i})\mid i\in\omega\}\mid
s\in \mathsf{Sub}
\}\right| 
\leq
\aleph_{0}. 
\end{equation}
For simplicity, we often write $\alpha$ or $\beta_{i}$ ($i\in\omega$) 
to represent any substitution instances, if there is no confusion.

\begin{example}(Tanaka \cite{tnk18}). 
With $\prmt$, $\prtp$, $\prcf$, and 
$\Box p\supset\Box\Box p$, 
the $\omega$-rule 
$$
\dfrac{p\supset\Diamond^{n}\top\ \text{($n\in\omega$)}}{p\supset\bot}
$$ 
axiomatizes the provability logic $\logic{GL}$. 
\end{example}

\begin{example}(Kaneko-Nagashima-Suzuki-Tanaka \cite{knk-ngs-szk-tnk}). 
Suppose that $\eko\phi$ and $\cko\phi$ denote that 
every agent knows $\phi$ 
and $\phi$
is common knowledge among the agents, respectively.  
Consider the following
$\omega$-rules: 
\begin{equation}\label{cklrules}
\dfrac{
\gamma
\supset
\Box_{1}(
\phi_{1}\supset
\Box_{2}(
\phi_{2}\supset
\cdots\supset
\Box_{k}(
\phi_{k}\supset
\eko^{n}\phi
)
\cdots
)
)
\text{\rm \hspace{10pt}($n\in\omega$) }
}
{
\gamma
\supset
\Box_{1}(
\phi_{1}\supset
\Box_{2}(
\phi_{2}\supset
\cdots\supset
\Box_{k}(
\phi_{k}\supset
\cko\phi
)
\cdots
)
)
},
\end{equation}
where $k\in\omega$ and each $\Box_{i}$ ($i=1,\ldots k$) is 
$\eko$ or $\cko$. 
Then, 
with $\prmt$, $\prtp$, $\prcf$, and axiom schemata 
$\cko p\supset\eko^{n}p$ ($n\in\omega$),
the $\omega$-rules \eqref{cklrules}
axiomatize the common knowledge logic. 
\end{example}

%%%%%%%%%%%%%%%%%%%%%%%%%%%%%%%%%%%%%%%%%%%%%%%%%%%%
%%%%%%%%%%%%%%%%%%%%%%%%%%%%%%%%%%%%%%%%%%%%%%%%%%%%
\section{Model existence theorem}\label{sec:mex}
%%%%%%%%%%%%%%%%%%%%%%%%%%%%%%%%%%%%%%%%%%%%%%%%%%%%
%%%%%%%%%%%%%%%%%%%%%%%%%%%%%%%%%%%%%%%%%%%%%%%%%%%%

In this section, we prove that for any predicate modal logic $\logic{L}$ 
that admits countably many pairs of a set of axiom schemata and an $\omega$-rule 
of the form \eqref{omega-rule}, there exists a neighborhood 
model $\nmodel{M}=\langle C,\nsystem,\ndom,\nint\rangle$ 
such that 
\begin{equation}\label{modelinfmeet}
c\vl_{\nmodel{M}}\alpha
\eq
c\vl_{\nmodel{M}}\beta_{i}\ \text{($i\in\omega$)} 
\end{equation}
for any $c\in C$.

Let $\logic{L}$ be a predicate modal logic. 
Define 
an equivalence relation $\lcong{L}$ on the set $\Phi$ of all predicate modal formulas by 
$\phi\lcong{L}\psi$ if and only if $\phi\equiv\psi\in\logic{L}$.  
For any formula $\phi\in\Phi$, 
we write $[\phi]$ for the equivalence class of $\phi$ in $\Phi/{\lcong{L}}$. 
It is easy to see that 
$\lind{L}=\langle \Phi/{\lcong{L}};\lor,\land,-,\Box,0,1\rangle$ 
is a modal algebra, 
in which operators $0$, $1$, $-$, $\lor$, $\land$, and $\Box$ on $\Phi/{\lcong{L}}$ 
are defined by the corresponding logical symbols of representatives. 
By the rule \eqref{congrule} in Definition \ref{pmlogic}, 
$\nec[\phi]:=[\nec\phi]$ defines a well-defined 
unary operator $\nec$ on  $\lind{L}$. 
It is easy to see that $\phi\supset\psi\in\logic{L}$ if and 
only if $[\phi]\leq[\psi]$ in $\lind{L}$, for every
formula $\phi$ and $\psi$ in $\logic{L}$.

\begin{lemma}\label{predlind}
Let $\logic{L}$ be a predicate modal logic that admits countably many 
pairs of axiom schemata and an $\omega$-rule of the form \eqref{omega-rule}. 
Then, $\lind{L}$ 
is a modal algebra that satisfies the following equations: 
\begin{equation}\label{larule}
[s(\alpha)]=\iand_{i\in\omega}[s(\beta_{i})],  
\end{equation}
where $s$ is any 
uniform substitution of formulas for
atomic formulas, 
and 
\begin{equation}\label{foralliand}
[\forall x\phi]=\iand_{y\in\variable}\big[ [y/x]\phi \big]
\end{equation}
for any formula $\phi\in\Phi$ and any $x\in\variable$. 
If $\logic{L}$
satisfies properties among $\prmt$, $\prtp$, and $\prcf$, 
then
$\lind{L}$
satisfies the properties corresponding to those of $\logic{L}$. 
\end{lemma}

\begin{proof}
We only show \eqref{larule}. 
Take any $i\in\omega$. 
Since $\alpha\supset\beta_{i}\in\logic{L}$, 
$$
[\alpha]\leq
[\beta_{i}]
$$
holds. 
Take any formula $\phi$ and 
suppose that $[\phi]\leq[\beta_{i}]$ for any $i\in\omega$. 
Then, $\phi\supset\beta_{i}\in\logic{L}$ for any $i\in\omega$. 
By the $\omega$-rule \eqref{omega-rule}, 
$\phi\supset\alpha\in\logic{L}$. Hence, $[\phi]\leq[\alpha]$. 
For the remainder of the proof, see \cite{tnk22}.  
\end{proof}

\begin{theorem}\label{mdlex}
Let $\logic{L}$ be a consistent predicate modal logic that 
admits countably many pairs of a set of axiom schemata and an $\omega$-rule 
of the form 
\eqref{omega-rule}.  
Then, there exists a neighborhood model 
$\nmodel{M}
=
\langle
C,\nsystem,\ndom,\nint
\rangle$ and an assignment $\nass $
such that 
\begin{equation}\label{nmodelomega}
\nval{\nint,\nass}(s\left(\alpha)\right)
=
\bigcap_{i\in\omega}
\nval{\nint,\nass}\left(s(\beta_{i})\right)
\end{equation}
for any uniform substitution  $s$,  
and 
\begin{equation}\label{canonicalmodel}
\phi\in\logic{L}\eq\nmodel{M}\vl\phi. 
\end{equation}
for any closed formula $\phi\in\Phi$. 
Moreover, if $\logic{L}$ satisfies properties 
among $\prmt$, $\prtp$, and $\prcf$
then 
the model satisfies 
the properties corresponding to those of $\logic{L}$.    
\end{theorem}

\begin{proof}
Let $\lind{L}$ be the Lindenbaum algebra of $\logic{L}$. 
Define $S_{1}$ and $S_{2}$ in $\power(\lind{L})$ by 
$$
S_{1}=
\left\{
\left\{
\left[s(\beta_{i})\right] 
\mid 
i\in\omega\right\}
\mid 
s\in\mathsf{Sub}
\right\}, \ 
S_{2}=
\left\{
\left\{
\left[\phi(y/x)\right] 
\mid 
y\in\variable\right\}\mid \forall x\phi\in\Phi\right\}, 
$$
respectively, and let $S=S_{1}\cup S_{2}$. 
Let $Z=\langle C,\nsystem\rangle$ be a 
Q-filter neighborhood frame of $\lind{L}$ for $S$, where  
$\nsystem$ is the neighborhood system given in Lemma \ref{canonicalsystem}
(if $A$ is monotonic, $\nsystem$ is defined by \eqref{monotonesystem}). 
Define 
an algebraic model 
$\nmodel{A}=
\langle
A,\ndom,\ninta
\rangle$
by 
$A=\dalg(Z)$, $\ndom=\variable$,  and
$$
P^{\ninta}(x_{1},\ldots,x_{n})=f([P(x_{1},\ldots,x_{n})])
$$ 
for any atomic formula $P(x_{1},\ldots,x_{n})$, 
where $f$ is an embedding given in Theorem \ref{extjt}.
It can be verified by induction on the construction of the formulas that 
\begin{equation}\label{lalgembedding}
u_{\ninta,\nass}(\phi)=f([\phi])
\end{equation}
for any formula $\phi$.  
The induction step for universal quantifiers follows from 
Theorem \ref{extjt} and \eqref{foralliand}. 
Take the neighborhood model 
$\nmodel{M}=\langle C,\nsystem,\ndom,\nint\rangle$
and the assignment $\nass$ given in Lemma \ref{amodeltonmodel}. 
First, we show \eqref{nmodelomega}: 
\begin{align*}
\nval{\nint,\nass}(\alpha)
&=
u_{\ninta,\nass}(\alpha)
&& \text{(Lemma \ref{amodeltonmodel})}\\
&=
f([\alpha])
&& \text{\eqref{lalgembedding}}\\
&=
f\left(\iand_{i\in\omega}[\beta_{i}]\right)
&& \text{\eqref{larule}}\\
&=
\bigcap_{i\in\omega}f\left([\beta_{i}]\right)
&& \text{\eqref{finfmeet}}\\
&=
\bigcap_{i\in\omega}u_{\ninta,\nass}\left(\beta_{i}\right)
&& \text{\eqref{lalgembedding}}\\
&=
\bigcap_{i\in\omega}\nval{\nint,\nass}\left(\beta_{i}\right). 
&& \text{(Lemma \ref{amodeltonmodel})}
\end{align*}
Next, we show \eqref{canonicalmodel}. Since $\phi$ is closed, 
\begin{align*}
\phi\in\logic{L}
&\eq
\text{$[\phi]=1$ in $\lind{L}$}\\
&\eq
\text{$f([\phi])=1$ in $\dalg(Z)$}
&& \text{($f$ is monomorphism)}\\
&\eq
\text{$u_{\ninta,\nass}(\phi)=1$ in $\dalg(Z)$}
&& \text{\eqref{lalgembedding}}\\
&\eq
\nval{\nint,\nass}(\phi)=C
&& \text{(Lemma \ref{amodeltonmodel})}\\
&\eq
\nmodel{M}\vl\phi. 
\end{align*}
For the remainder of the proof, see \cite{tnk22}.  
\end{proof}

\begin{corollary}\label{completeness}
If there exists a neighborhood system $\nsystem$ such that 
the Q-filter neighborhood frame 
$\dfrm_{S,\nsystem}(\lind{L})\in\deffrm(\logic{L})$,
then 
$\logic{L}$ is sound and complete with respect to $\deffrm(\logic{L})$. 
\end{corollary}

\begin{proof}
Soundness follows directly from the definition of $\deffrm(\logic{L})$. 
We show completeness. Suppose that $\phi\not\in\logic{L}$ and $\phi$ is closed. 
Then, $\dfrm_{S,\nsystem}(\lind{L})\not\vl\phi$, as is shown in the proof of Theorem \ref{mdlex}. 
Then, $\deffrm(\logic{L})\not\vl\phi$, 
since $\dfrm_{S,\nsystem}(\lind{L})\in\deffrm(\logic{L})$.  
\end{proof}

%%%%%%%%%%%%%%%%%%%%%%%%%%%%%%%%%%%%%%%%%%%%%%%%%%%%
%%%%%%%%%%%%%%%%%%%%%%%%%%%%%%%%%%%%%%%%%%%%%%%%%%%%
\section{Completeness of a predicate 
extension of $\logic{GL}$}\label{sec:glcompleteness}
%%%%%%%%%%%%%%%%%%%%%%%%%%%%%%%%%%%%%%%%%%%%%%%%%%%%
%%%%%%%%%%%%%%%%%%%%%%%%%%%%%%%%%%%%%%%%%%%%%%%%%%%%

In this section, we present a predicate extension of the
logic $\logic{GL}$ of provability
which is sound and complete with respect 
to a class of neighborhood frames with constant domains. 
Define the predicate modal logic $\qglo$ 
to be the set of all formulas that are derivable 
in the following proof system $\psqglo$:

\begin{definition}
$\psqglo$ 
consists of the following axiom schemata and inference rules: 
\begin{enumerate}
\item
all tautologies of classical predicate logic;
\item
$\Box(p\supset q)\supset(\Box p\supset\Box q)$;
\item
$\Box p\supset\Box\Box p$;
\item
modus ponens;
\item
uniform substitution rule;
\item
necessitation rule for the modal operator;
\item
generalization rule; 
\item
$
\dfrac{p\supset\Diamond^{n}\top\ \text{($n\in\omega$)}}{p\supset\bot} 
$ ($\Diamond^{\ast}$).
\end{enumerate}
\end{definition}

The proof system $\psqglo$ is equivalent to a Gentzen-style proof system 
$\nqgl$ given in \cite{tnk18}. 
More precisely, if $\phi$ is provable in $\psqglo$ then 
the sequent $\yy\phi$ is provable in $\nqgl$, and if a sequent 
$\gm\yy\dl$ is provable 
in $\nqgl$ then $\iand\gm\supset\ior\dl$ is provable in $\psqglo$. 
It is shown in \cite{tnk18} that 
the propositional fragment of $\nqgl$ coincides with 
$\logic{GL}$, and that $\nqgl$ is sound and complete with respect to the 
class of Kripke frames of locally finite height with expanding domains.  
Consequently, $\psqglo$ inherits these properties. 
Therefore, $\qglo$ is a proper superset of the minimal predicate 
extension $\logic{QGL}$ of $\logic{GL}$, since
$\logic{QGL}$ is not complete with 
respect to any class of Kripke frames
(Montagna \cite{mnt84}). 
In the rest of the section, we show that $\qglo$ is sound and complete 
with respect to the following 
class $\cglo$ of neighborhood frames with constant domains.

\begin{definition}
A neighborhood frame 
$Z=\langle C,\nsystem\rangle$ is called a {\em GL$_{\Diamond_{\ast}}$-frame} if 
$\dalg(Z)$ satisfies 
$\prmt$, $\prtp$, $\prcf$, as well as the following 
additional conditions: 
\begin{enumerate}
\item
$\Box_{Z} X\subseteq\Box_{Z}\Box_{Z} X$ for every $X\in\power(C)$;
\item
$\bigcap_{n\in\omega}{\Diamond_{Z}}^{n}C=\emptyset$.
\end{enumerate}
We write $\cglo$ for the class of all GL$_{\Diamond_{\ast}}$-frames. 
\end{definition}

\begin{theorem}
$\qglo$ is sound and complete with respect to $\cglo$. 
\end{theorem}

\begin{proof}
Soundness is proved by induction on the height of 
the derivations in $\psqglo$. 
Hence, $\cglo\subseteq\deffrm(\qglo)$. 
We prove completeness.  
By Corollary \ref{completeness}, it suffices to show that 
$\dfrm_{S,\nsystem}\left(A(\qglo)\right)\in\cglo$, 
where $A(\qglo)$ is the Lindenbaum algebra of $\qglo$ 
and $\nsystem$ is 
the neighborhood system defined by \eqref{monotonesystem}. 
First, we show  that $\bigcap_{n\in\omega}{\Diamond_{Z}}^{n}C=\emptyset$.
Since 
$
f\colon  A(\qglo)\yy \dalg\left(\dfrm_{S,\nsystem}\left(A(\qglo)\right)\right)
$
given in Theorem \ref{extjt} preserves the infinite meet of the $\omega$-rule in 
$\psqglo$, 
$$
\bigcap_{n\in\omega}{\Diamond_{Z}}^{n}C
=
\bigcap_{n\in\omega}{\Diamond_{Z}}^{n}f([\top])
=
f\left(\iand_{n\in\omega}\Diamond^{n}[\top]\right)
=
f(\bot)
=
\emptyset. 
$$ 
Next, we show that $\Box_{Z} X\subseteq\Box_{Z}\Box_{Z} X$ for 
any $X\in \dalg\left(\dfrm_{S,\nsystem}\left(A(\qglo)\right)\right)$.
Take any $F\in \Box_{Z} X$. Then, there exists a formula $\Box\phi$ such that 
$\Box[\phi]\in F$ and $f([\phi])\subseteq X$. 
Hence, by $\prmt$, 
\begin{equation}\label{boxbox}
f\left(\Box[\phi]\right)
=
\Box_{Z} f\left([\phi]\right)
\subseteq
\Box_{Z} X. 
\end{equation}
Since $\Box\phi\supset\Box\Box\phi$ is in $\qglo$, 
$\Box\Box[\phi]\in F$. 
Therefore, $F\in\Box_{Z}\Box_{Z} X$ by \eqref{boxbox}. 
\end{proof}

%%%%%%%%%%%%%%%%%%%%%%%%%%%%%%%%%%%%%%%%%%%%%%%%%%%%
%%%%%%%%%%%%%%%%%%%%%%%%%%%%%%%%%%%%%%%%%%%%%%%%%%%%
\section{Kripke incompleteness of a common knowledge logic}\label{sec:cklincompleteness}
%%%%%%%%%%%%%%%%%%%%%%%%%%%%%%%%%%%%%%%%%%%%%%%%%%%%
%%%%%%%%%%%%%%%%%%%%%%%%%%%%%%%%%%%%%%%%%%%%%%%%%%%%

In this section, we prove that there exists 
a common knowledge logic that 
is neighborhood complete but Kripke incomplete. 
We define that a bimodal logic $\logic{L}$ 
with 
two modal operators $\eko$ and $\cko$ 
is a 
{\em common knowledge logic} if it satisfies the following properties: 
\begin{enumerate}
\item
for any formula $\phi$ and any $n\in\omega$, 
$\cko\phi\supset\eko^{n}\phi\in\logic{L}$; 

\item
for any formulas $\phi$ and $\psi$, 
if
$\psi\supset\eko^{n}\phi\in\logic{L}$ for any $n\in\omega$, 
then 
$\psi\supset\cko\phi\in\logic{L}$. 
\end{enumerate}

We define a common knowledge logic 
$\logic{QCKL}$ to be the set of all formulas that are derivable 
in the following proof system $\psqckl$:  

\begin{definition}\label{psqckl}
The proof system $\psqckl$ 
consists of the following axiom schemata and inference rules: 
\begin{enumerate}
\item
all tautologies of classical predicate logic;
\item
$\Box(p\supset q)\supset(\Box p\supset\Box q)$, where $\Box=\eko$ or $\Box=\cko$;
\item
for any $n\in\omega$, 
$\cko p\supset\eko^{n} p$; 
\item
$\barcan$, where $\Box=\eko$ or $\Box=\cko$;
\item
modus ponens;
\item
uniform substitution rule;
\item
necessitation rule for the modal operators $\eko$ and $\cko$;
\item
generalization rule; 
\item \label{cklright}
$
\dfrac{
\gamma
\supset
\Box_{1}(
\phi_{1}\supset
\Box_{2}(
\phi_{2}\supset
\cdots\supset
\Box_{k}(
\phi_{k}\supset
\eko^{n}\phi
)
\cdots
)
)
\text{\rm \hspace{10pt}($n\in\omega$) }
}
{
\gamma
\supset
\Box_{1}(
\phi_{1}\supset
\Box_{2}(
\phi_{2}\supset
\cdots\supset
\Box_{k}(
\phi_{k}\supset
\cko\phi
)
\cdots
)
)
},
$
for each $k\in\omega$ and $\{\Box_{i}\mid 1\leq i\leq k\}\subseteq\{\eko,\cko\}$. 
\end{enumerate}
The set of premises of each of the inference rules (\ref{cklright}) is countable. 
When $k=0$, 
\eqref{cklright} means that 
\begin{equation}\label{cright1}
\dfrac
{\gamma\supset\eko^{n}\phi\text{\rm \hspace{10pt}($n\in\omega$) }}
{\gamma\supset\cko\phi} .
\end{equation}
\end{definition}

The $\omega$-rules \eqref{cklright} of Definition \ref{psqckl} are introduced
by Kaneko-Nagashima-Suzuki-Tanaka \cite{knk-ngs-szk-tnk}.
The system $\psqckl$ is equivalent to $\psck$, a Gentzen-style proof
system given by Tanaka \cite{tnk01}. 
It is proved  that $\psck$ is sound and 
complete with respect to the class of Kripke frames with constant domains 
such that $R_{\cko}=\bigcup_{n\in\omega}{R_{\eko}}^{n}$ (\cite{tnk01}). 
Therefore, 
$\logic{QCKL}$ is Kripke complete. 
We claim that removing the Barcan formula and 
replacing \eqref{cklright} of Definition \ref{psqckl} 
with \eqref{cright1} causes 
the resulting proof system $\psqcklm$ to be 
neighborhood complete but Kripke incomplete.

\begin{definition}
The proof system $\psqcklm$ 
consists of (1)-(3) and (5)-(8) of Definition \ref{psqckl}
and \eqref{cright1}. 
Define the logic $\qcklm$ to be the set of all formulas that are derivable 
in $\psqcklm$. 
\end{definition}

\begin{definition}
A complete modal algebra $A$ with two modal operators $\eko$ and $\cko$ is 
called a {\em CKL$^{-}$-algebra}, if it satisfies $\prmt$, $\prtp$, $\prcf$, and 
$$
\cko x=\iand_{n\in\omega}\eko^{n} x
$$ 
for any $x\in A$. 
We write $\acklm$ for the class of all CKL$^{-}$-algebras. 
A neighborhood frame $Z=\langle C,\nsystem_{\eko},\nsystem_{\cko}\rangle$
is called a {\em CKL$^{-}$-frame} if
$\dalg(Z)\in\acklm$. 
We write $\ccklm$ for the class of all CKL$^{-}$-frames. 
\end{definition}

\begin{theorem}\label{qcklmncomp}
$\qcklm$ is sound and complete with respect to $\ccklm$. 
\end{theorem}

\begin{proof}
Soundness is proved by induction on the height of the derivations in the proof system 
$\psqcklm$. 
We show completeness. 
Let $A$ be the Lindenbaum algebra of $\qcklm$. 
Then $A$ is in $\acklm$ by Lemma \ref{predlind}. 
Define subsets $S_{1}$ and $S_{2}$ of $\power(A)$ by 
$$
S_{1}=
\left\{
\left\{
\left[\eko^{n}\phi\right] 
\mid 
n\in\omega\right\}
\mid 
\phi\in\Phi
\right\}, \ 
S_{2}=
\left\{
\left\{
\left[\phi(y/x)\right] 
\mid 
y\in\variable\right\}\mid \forall x\phi\in\Phi\right\}, 
$$
respectively, and let $S=S_{1}\cup S_{2}$. 
Define 
the neighborhood frame $Z=\langle C,\nsystem_{\eko},\nsystem_{\cko}\rangle$ 
in the same way as in the proof of Theorem \ref{mdlex}, except for $\nsystem_{\cko}$. 
We define the neighborhood system $\nsystem_{\cko}$ by 
$$
\nsystem_{\cko}(F)
=
\left\{
X\mid
F\in \bigcap_{n\in\omega}\eko^{n}X
\right\}
$$
for any $F\in\qflt{S}(A)$. 
Then, $Z$
is a CKL$^{-}$-frame, since 
\begin{align*}
F\in\cko X
\eq
X\in\nsystem_{\cko}(F)
\eq
F\in\bigcap_{n\in\omega}\eko^{n}X
\end{align*}
for any $F\in\qflt{S}(A)$ and $X\subseteq C$. 
We check that
$\nsystem_{\cko}$ satisfies \eqref{dfrmcond}. 
Take any $x\in A$. 
\begin{align*}
x\in\cko^{-1}[F]
&\eq
\cko x\in F\\
&\eq
\eko^{n} x\in F \ \text{($\forall n\in\omega$)}\\
&\eq
F\in  \eko^{n}f(x) \ \text{($\forall n\in\omega$)}\\
&\eq
F\in  \bigcap_{n\in\omega}\eko^{n}f(x)\\
&\eq
x\in  \bigcup_{X\in\nsystem_{\cko}(F)}\bigcap X. 
\end{align*}
The if-part of the last equivalence follows, since 
$X\in\nsystem_{\cko}$ and 
$x\in\bigcap X$ implies that $F\in\bigcap_{n\in\omega}\eko^{n}X$ and $X\subseteq f(x)$. 
The only if-part follows, since $f(x)\in\nsystem_{C}(F)$. 
Since $Z$ is a CKL$^{-}$-frame, $Z\in\deffrm(\qcklm)$ by soundness. 
By Corollary \ref{completeness}, 
$\qcklm$ is complete with respect to $\ccklm$. 
\end{proof}

\begin{corollary}\label{qcklmacomp}
$\qcklm$ is sound and complete with respect to $\acklm$. 
\end{corollary}

\begin{proof}
Soundness is proved by induction on 
the height of the derivations in $\psqcklm$. 
Completeness follows since the Lindenbaum algebra  
of $\qcklm$ is in $\acklm$. 
\end{proof}

\begin{corollary}\label{corcebarcan}
The formula 
\begin{equation}\label{cebarcan}
\cko p\supset \eko\cko p
\end{equation}
is not in $\qcklm$. 
\end{corollary}

\begin{proof}
By Corollary \ref{qcklmacomp}, it is enough to show that 
there exists a CKL$^{-}$-algebra $A$ that does not validate \eqref{cebarcan}. 
Let $2$ be the two-element Boolean algebra. 
Define $A=2^{\omega+\omega}$. 
We write $1_{A}$ and $0_{A}$ for the greatest and the least elements of $A$, 
respectively. 
For each $x\in A$, define the set $N_{x}\subseteq\omega+\omega$ by 
$$
N_{x}=\{\alpha\mid x(\alpha)=0\},    
$$
and for each $x\in A$ such that $N_{x}$ is not cofinal in $\omega+\omega$, 
define $n_{x}\in\omega+\omega$ by 
$$
n_{x}=\min\{\alpha\mid \forall \beta\geq\alpha\left(x(\beta)=1\right)\}. 
$$
Define the unary operator $\eko$ on $A$ as follows: 
$\eko 1_{A}=1_{A}$; 
if $N_{x}$ is cofinal in $\omega+\omega$, then $\eko x=0_{A}$; 
otherwise, 
$$
\eko x(\alpha)=
\begin{cases}
0 & \text{($\alpha\leq n_{x}$)}\\
1 & \text{(otherwise)} 
\end{cases}.
$$
Define 
$\cko x=\iand_{n\in\omega}\eko^{n}x$ for any $x\in A$. 
We claim that $A$ is a CKL$^{-}$-algebra. 
First, we show that $\eko(x\land y)=\eko x\land \eko y$ for any $x$ and $y$ in $A$. 
The case that $x=1_{A}$ or $y=1_{A}$ is straightforward. 
Suppose that $N_{x}$ or $N_{y}$ is 
cofinal in $\omega+\omega$. Then, $N_{x\land y}$ is also cofinal in 
$\omega+\omega$. Therefore, 
$$
\eko(x\land y)=0_{A}=\eko x\land \eko y. 
$$
Suppose otherwise. 
Then, $N_{x\land y}$ is not cofinal in $\omega+\omega$ and 
$n_{x\land y}=\max\{n_{x},n_{y}\}$. Hence, 
$$
\eko(x\land y)(\alpha)=\eko x(\alpha)\land\eko y(\alpha)
=
\begin{cases}
0 & \text{($\alpha\leq \max\{n_{x},n_{y}\}$)}\\
1 & \text{(otherwise)}
\end{cases}. 
$$
Next, we prove
that $\cko(x\land y)=\cko x\land \cko y$ for any $x$ and $y$ in $A$. 
We only show the case that both $N_{x}$ and $N_{y}$ are not cofinal 
in $\omega+\omega$. If $n_{x}\geq\omega$ or $n_{y}\geq\omega$, 
then $n_{x\land y}\geq\omega$. 
Hence, 
$$
\cko(x\land y)=0_{A}=\cko x\land \cko y. 
$$
If not, $n_{x\land y}\in\omega$ and 
$$
\cko(x\land y)(\alpha)=\cko x(\alpha)\land\cko y(\alpha)
=
\begin{cases}
0 & \text{($\alpha<\omega$)}\\
1 & \text{(otherwise)}
\end{cases}. 
$$
This completes the proof of the claim. 
We show that $A$ does not satisfy \eqref{cebarcan}. 
Let $x\in A$ be  
$$
x(\alpha)=
\begin{cases}
0 & \text{($\alpha=1$)}\\
1 & \text{(otherwise)}
\end{cases}. 
$$
Then, 
$$
\cko x(\alpha)=
\begin{cases}
0 & \text{($\alpha<\omega$)}\\
1 & \text{(otherwise)}
\end{cases}, 
\hspace{10pt}
\eko\cko x(\alpha)=
\begin{cases}
0 & \text{($\alpha<\omega+1$)}\\
1 & \text{(otherwise)}
\end{cases}. 
$$
\end{proof}

Next, we show the Kripke incompleteness of $\qcklm$.

\begin{lemma}\label{cklkripke}
Let $Z$ be a Kripke frame. 
If $Z\vl\qcklm$ then: 
\begin{enumerate}
\item
$\cko X=\bigcap_{n\in\omega}\eko^{n} X$ for any $X\in \dalg(Z)$; 
\item
$Z\vl \cko p\supset\eko\cko p$. 
\end{enumerate}
\end{lemma}

\begin{proof}
(1): 
We first prove that the formula 
\begin{equation}\label{mhformula}
\cko (p\supset\eko p)\supset (p\supset\cko p)
\end{equation}
is derivable in $\psqcklm$. 
By the $\omega$-rule \eqref{cright1}, 
it is sufficient to show that for each $k\in\omega$,  
\begin{equation}\label{lemmainduction}
\left(p\land\cko (p\supset\eko p)\right)
\supset 
\eko^{k} p 
\end{equation}
is derivable in $\psqcklm$. 
We prove \eqref{lemmainduction} by induction on $k\in\omega$. 
The base step is trivial.  
Suppose that \eqref{lemmainduction} is derivable for every $k\leq m$. 
By (3) of Definition \ref{psqckl}, 
\begin{equation}\label{lemmainduction2}
\left(p\land\cko (p\supset\eko p)\right)
\supset 
\eko^{m} (p\supset\eko p)
\end{equation}
is derivable in $\psqcklm$. Hence, by the induction hypothesis and 
\eqref{lemmainduction2},
\begin{equation*}
\left(p\land\cko (p\supset\eko p)\right)
\supset 
\eko^{m+1} p 
\end{equation*}
is derivable in $\psqcklm$. 
This proves \eqref{lemmainduction}. Hence, 
\eqref{mhformula}
is derivable in $\psqcklm$. 
Let $X\in \dalg(Z)$ and $Y=\bigcap_{n\in\omega}\eko^{n}X$. 
Take a propositional variable  $p$  and an interpretation $\nint$ on $Z$
such that 
$\nint(p)=Y$. 
Since $\qcklm$ is valid in $\dalg(Z)$ and includes the 
formula \eqref{mhformula}, 
\begin{equation}\label{algmh}
\cko(Y\supset \eko Y)\subseteq Y\supset\cko Y. 
\end{equation}
Since $Z$ is a Kripke frame, $\dalg(Z)$ is completely multiplicative. Hence, 
$$
Y
=
\bigcap_{n\in\omega}\eko^{n}X
\subseteq
\bigcap_{n\in\omega}\eko^{n+1}X
=
\eko\bigcap_{n\in\omega}\eko^{n}X
=
\eko Y.
$$
Hence, $Y\supset\eko Y$ is the top element of $\dalg(Z)$. 
By \eqref{algmh}, 
we have $Y\subseteq \cko Y$. Hence, 
$$
\bigcap_{n\in\omega}\eko^{n}X
=
Y
\subseteq
\cko Y
=
\cko\bigcap_{n\in\omega}\eko^{n}X
\subseteq
\cko X. 
$$
The converse is straightforward.

\noindent 
(2): For any $X\in \dalg(Z)$, we have 
$\cko X\subseteq\eko\cko X$,  
by (1) and complete multiplicativity of $\dalg(Z)$. 
Hence, 
$Z\vl \cko p\supset\eko\cko p$. 
\end{proof}

\begin{theorem}
$\qcklm$ is Kripke incomplete. 
\end{theorem}

\begin{proof}
Corollary \ref{corcebarcan} and Lemma \ref{cklkripke}. 
\end{proof}

It is straightforward to adapt the preceding argument to 
show that the propositional fragment of $\qcklm$ is neighborhood 
complete but Kripke incomplete. 

\bibliographystyle{plain}
\bibliography{myref.sjis}
\end{document}